\documentclass[11pt,reqno]{amsart}
\usepackage[a4paper,left=22mm,right=22mm,top=35mm,bottom=37mm,marginpar=20mm]{geometry} 

\usepackage{amsmath,amsthm,amssymb,amsfonts,epsfig,graphicx,enumerate,enumitem,accents}
\usepackage{rotating,multirow,color,enumerate,url,mathrsfs,mathrsfs}
\usepackage{arydshln}
\usepackage{bbm}
\usepackage{cases}
\usepackage{esint}
\usepackage{epstopdf}
\usepackage{comment}
\usepackage{leftidx}
\usepackage{graphicx}
\usepackage{todonotes}
\usepackage[normalem]{ulem}
\usepackage{mathtools}
\usepackage{enumerate}
\usepackage{subfig}
\usepackage{amsaddr}
\usepackage[rightcaption]{sidecap}

\usepackage[utf8]{inputenc}

\usepackage{scalerel}
\usepackage{bbm}
\usepackage{tikz}
\usepackage{IEEEtrantools}
\usepackage{thmtools}
\usepackage{thm-restate}

\newtheorem{theorem}{Theorem}[section]
\newtheorem{corollary}[theorem]{Corollary}
\newtheorem{lemma}[theorem]{Lemma}

\theoremstyle{definition}

\newtheorem{remark}[theorem]{Remark}
\newtheorem{example}[theorem]{Example}

\numberwithin{equation}{section}

\newcommand{\med}{\medskip\noindent}

\newcommand{\R}{\mathbb{R}}

\newcommand{\Rd}{{\R^d}}

\newcommand{\ov}{\overline}

\newcommand{\var}{{\rm{var}}}

\newcommand{\opnorm}[1]{{\left\vert\kern-0.25ex\left\vert\kern-0.25ex\left\vert #1 
		\right\vert\kern-0.25ex\right\vert\kern-0.25ex\right\vert}}

\newcommand{\argu}{{\,\cdot\,}}

\newcommand{\PP}{\mathcal{P}}

\newcommand{\pairing}[1]{{\left \langle #1 \right \rangle}}
\newcommand{\norm}[1]{\Arrowvert #1 \Arrowvert}
\newcommand{\abs}[1]{{\left \lvert #1 \right \rvert}}
\newcommand{\Lip}{{\mathrm{lip}}}

\newcommand{\eps}{\varepsilon}

\newcommand{\V}{{\mathcal{V}}}

\newcommand{\one}{{{\bf 1}
		\kern-0,28em \rm l}}

\DeclareMathOperator{\spt}{sp}

\def\dis{\displaystyle}

\begin{document}

	\title{ Sharp inequalities between Zolotarev and Wasserstein distances in $\PP_2(\R^d)$ }

	

	\author{Karol Bo{\l}botowski \and Guy Bouchitt\'{e}}


	\date{\today}
	
	\begin{abstract}  
		Based on a new Kantorovich-Rubinstein duality principle for the Hessian that was recently established by the two authors,
		 we extend the Rio inequality to any dimension $d \ge  \nolinebreak 1$ with an optimal constant.  Similarly, we propose an optimal upper bound for the ratio  of Zolotarev distance $Z_2(\mu,\nu)$ to Wasserstein distance $W_2(\mu,\nu)$ when  $\mu,\nu \in \PP_2(\R^d)$  are centred probabilities with prescribed variances.  
 		\end{abstract}

\maketitle

\noindent \textbf{Keywords:}  Probability metrics,  Optimal transport, Kantorovich-Rubinstein duality,  Wasserstein distance, Zolotarev distance, Rio inequality.

\noindent \textbf{2020 Mathematics Subject Classification:}   26D10, 49K30, 60E05, 60E15.
	
	\dedicatory{}
	
	
	\maketitle

	\vskip1cm

\section{introduction}

A considerable amount of literature is devoted to constructing metric structures on sets of measures. 
 It begins with total variation (TVD) and then, based on the needs of statistics and information theory, moves on to the pioneering works of Kullback-Leibler divergence (KL divergence) and Hellinger-Bhattacharyya distance.
More recently, Monge-Kantorovich optimal transport theory gave rise to transportation metrics, such as the Prokhorov distance (see for instance \cite{Dudley}) and the Wasserstein distances \cite{villani2003,santambrogio2015}. In this vein, V.M. Zolotarev introduced in the 1970s the vast class of   so-called ideal metrics \cite{zolotarev1978}.  This class includes TVD and the first Wasserstein metric.

This topic remains active today, providing a fundamental framework for several areas of research. Notably, it is used in mathematical physics to treat random measures as random variables in a Polish metric space. In statistics, it is used to evaluate rates of convergence in the Central Limit Theorem and to establish a large deviation principle for interacting particle models. It is also used in the promising field of machine learning.



In $\Rd$, or more generally in a complete metric space, the Wasserstein distance $W_p$ is amongst the most popular choices for a metric between two elements of $\PP_p(\Rd)$ (probabilities of finite $p$-th moment), where $p\in [1,+\infty]$.
It is expressed through the Monge-Kantorovich optimal transport problem: if $p< +\infty$, then
\begin{equation*}
	W_p(\mu,\nu) = \left( \inf\left\{ \iint \abs{x-y}^p \gamma(dxdy) \, : \, 	\gamma \in \Gamma(\mu,\nu)   \right\} \right)^{\frac{1}{p}},
\end{equation*}
where $\Gamma(\mu,\nu)$ is the set of admissible transport plans, i.e. elements of $\PP(\Rd \times \Rd)$ whose first and second marginals are $\mu$ and $\nu$, respectively.  Arguably, the most prominent cases of the Wasserstein distance are $p = 1,2$. By virtue of the Kantorovich-Rubinstein duality, the Monge distance $W_1$ can be reset as the supremum with respect to 1-Lipschtitz potentials,
\begin{equation}
	\label{W1}
	W_1(\mu,\nu) = \sup \left\{ \int u \, d(\nu-\mu) \ : \ u \in C^{0,1}(\Rd), \ \ \Lip(u) \le 1 \right\} .
\end{equation}
With his seminal paper  \cite{brenier1991}, Y. Brenier has established the fundamental importance of the quadratic case $W_2$. He  has proved that the optimal plans are supported on  cyclically monotone subset of $(\Rd)^2$. Furthermore, under the condition of absolute continuity of one of the marginals, the unique optimal plan is induced by a transport map being the gradient of a convex potential.

Another family of metrics between probabilities on $\Rd$ that we  are interested in is the one proposed by V.M. Zolotarev \cite{zolotarev1978}. For an integer\footnote{The definition of $Z_r$ can be extended to any real order $r >0$: the right hand side of the two-point inequality must be then changed to $\abs{x-y}^\alpha$ for $\alpha = r+1 - \lceil r \rceil$.} $r \geq 1$, the Zolotarev distance generalizes the Monge distance \eqref{W1} to higher orders,
\begin{equation*}
	Z_r(\mu,\nu) = \sup\left\{ \int \! u \, d(\nu-\mu) \, : \, u \in C^{r-1,1}(\Rd), \   \norm{D^{r-1}u(x) - D^{r-1}u(y)}  \leq \abs{x-y} \ \ \forall\, x,y \in \Rd \right\},
\end{equation*}
where $D^k u$ is the $k$-th Fréchet derivative and $\norm{\argu}$ stands for the operator norm.
Clearly, $Z_1 = W_1$. Note that the two-point inequality can be equivalently replaced by $\norm{D^r u} \leq 1$ enforced a.e. in $\Rd$. The Zolotarev distance $Z_r$ is an example of \textit{ideal metric} of order $r$: it is scale-$r$-homogenous and it is subadditive under convolution. These properties render $Z_r$ useful in obtaining Berry-Esseen-type  bounds \cite{esseen1958} in the Central Limit Theorem, see Chapters 14,\,15 in \cite{rachev1991} for more details.



In the short note \cite{belili2000} it was shown that the metrics $Z_r$ and $W_r$ induce the same topology on $\PP_r(\Rd)$: the one of weak convergence together with convergence of $r$-th moments\footnote{More accurately, the equivalence between the $Z_r$ and $W_r$ metrics holds true only on the relevant subspaces of $\PP_r(\Rd)$ with fixed mixed moments of order less than $r$, cf. Remark \ref{finiteZ} for the case $r=2$.}. Looking for quantitative results comparing the two families of distances is an active field of study. For instance, since $Z_r$ and $(W_r)^r$  scale identically with respect to dilations, the inequalities of the kind $Z_r \ge  C_r  W_r^r$ can be pursued. Thus far, however, the results in this direction are mainly restricted to one dimension $d=1$, where we can utilize the closed solutions for $W_r$ that are expressed via the quantile functions of $\mu,\nu$. The first result on the real line can be traced back to the works of E. Rio, cf. \cite{Rio1998}. The constants $C_r$ were then improved in \cite{rio2009}. For certain orders $r$, an explicit formula for the constant was given. For instance, the inequality  $Z_2 \ge  C_2  W_2^2$ holds true with $C_2 = \frac{1}{4}$. Similar inequalities in 1D were established in \cite{bobkov2023} for a modified version of Zolotarev distance where a uniform bound on the lower-order derivatives $u^{(k)}$ is included in the defining supremum.

In this paper we take advantage of the  novel duality method, recently developed by the two authors in \cite{bolbou2024}, to extend the validity of Rio's inequality for $r=2$,
\begin{equation}\label{Rio-opt}
Z_2(\mu,\nu) \ge  \frac1{4} W_2^2(\mu,\nu),
\end{equation}
to any dimension $d \geq 1$, showing that the constant $\frac1{4}$ is optimal and that the inequality is strict unless $\mu=\nu$ (see Theorem \ref{main}). 

On the other hand,  as we will show, the reverse inequality of the same form cannot hold true for any finite constant. However, the topologies induced by the $Z_2$ and $W_2$ distances are equivalent.
By exploiting the same duality technique, we provide a quantitative confirmation  of this topological equivalence by showing a modified  reverse inequality under the assumption that the barycentres of $\mu$ and $\nu$ coincide,
$$ Z_2(\mu,\nu)\, \le \frac{1}{2} (\sigma_\mu +\sigma_\nu)  \, W_2(\mu,\nu) ,$$
where $\sigma_\mu$, $\sigma_\nu$ are the respective standard deviations of $\mu, \nu$ (see  Theorem \ref{thm_upper}). The factor $\frac{1}{2} (\sigma_\mu +\sigma_\nu)$ is optimal in the sense to be made precise. 
As a by-product, we deduce in Corollary \ref{variant_upper} an alternative version of the upper bound inequality involving the variances and where the inequality  is strict unless $\mu=\nu$,
\begin{equation}\label{new-upper}
Z_2(\mu,\nu) \, \leq \ \sqrt{\frac{\var\, \mu + \var\, \nu}{2}}\ W_2(\mu,\nu) .
\end{equation}
It is worth to note that similar reverse inequalities were proposed in \cite{bobkov2023} for the modified Zolotarev distance, again in one dimension only.

\medskip
To conclude the introduction, let us briefly mention one of the possible applications of the new inequalities, which  can be traced back to the paper \cite{rio2009}. Therein, Rio has derived his inequality to establish optimal convergence rates in the univariate Central Limit Theorem with respect to  $W_2$ distance. More precisely, with $\mu_n$ being the law of the sum of iid random variables of zero mean and unit variance $n^{-1/2} \sum_{i=1}^n X_n$, and with $\nu$ being the standard normal distribution, he has shown that $W_2(\mu_n,\nu)$ decays as $n^{-1/2}$ provided that $X_1$ has a finite forth moment. 

 Similarly, our multidimensional extension \eqref{Rio-opt} could potentially pave the way towards finding new rates for the multivariate case of CLT in Wasserstein distance, see \cite{bonis2020,bonis2024} for the state of the art. Similarly as in \cite{rio2009}, the strategy would be to work on the rates in $Z_2$ instead of directly in $W_2$.  One way of estimating  the Zolotarev distance in CLT is by employing  the Stein method \cite{stein1986}, see also the survey \cite{ross2011}. For instance, the paper \cite{fathi2021} has put forth the notion of higher-order Stein kernels which has allowed the author to provide $n^{-r/2}$ decay rates for $Z_r$ under  additional moment constraints and a Poincaré inequality condition. In \cite{gaunt2020} the Stein method was applied in the context of the smooth Wasserstein distance which is reminiscent of the Zolotarev distance.

\medskip
The  paper is organized as follows: in the preliminary Section  2,  we present the duality scheme introduced in \cite{bolbou2024} for the $Z_2$ distance and its connection with convex order and the three marginal optimal transport formulation that involves martingale conditions. Then, we deduce an identity that will be crucial in establishing the upper bound inequality (see Lemma \ref{magic_formula}).
In Section 3, we prove a sharp  generalization of the Rio lower bound inequality to any dimension (see Theorem \ref{main}). In the last Section 4, we establish a new sharp upper bound inequality involving  standard deviations
(see Theorem \ref{thm_upper}) and its variant in Corollary \ref{variant_upper}.

\subsection*{Acknowledgments}

This paper was conceived during the first author's visits to the University of Toulon and the second author's visits to the Lagrange Mathematics and Computing Research Center in Paris.  We are grateful to the Lagrange Center for funding these visits. We would also like to thank G. Carlier, Q. Merigot, and F. Santambrogio for their valuable discussions, which greatly influenced the development of the probabilistic section of the paper \cite{bolbou2024}, from which this work benefits. During the preparation of this paper, the first author was a postdoc at the Lagrange Center.
The second author was partially supported by the French project ANR-23-CE40-0017.


\subsection*{Notations} Throughout the paper we will use the following notations.

\begin{itemize}[leftmargin=1.2\parindent]
\item  The Euclidean norm of $z\in\R^d$ is denoted by $|z|$, and  $\pairing{\argu,\argu}$ will be used to denote the canonical scalar product.

\item  By $C^{0,1}(\Rd)$ (resp. $C^{1,1}(\Rd)$) we understand the Banach space of these continuous functions $u \in C(\Rd)$  (resp. continuously differentiable functions  $u\in C^1(\Rd)$) for which $\Lip(u)<+\infty$ (resp. $\Lip(\nabla u) <+\infty$).



 \item  $\PP(\Rd)$ is the set of probabilities on $\Rd$, i.e. positive Borel measures $\mu$ of unit total mass, $\mu(\Rd) = 1$.  For any $p\ge 1$, $\PP_p(\Rd)$ denotes the subset of whose elements of $\PP(\Rd)$
 which satisfy $\int |x|^p\, d\mu <+\infty$.
 \item  $\delta_{x_0}$ is the Dirac delta measure at the point  $x_0 \in \Rd$.
 
 \item For a probability $\mu$ and a $\mu$-measurable map $T$, by  $T^\#(\mu)$ we understand the push forward, i.e.  $T^\#(\mu)(B) := \mu\big( T^{-1}(B) \big)$ for every Borel set $B$.

  \item The topological support of $\mu \in \PP(\Rd)$ is denoted  by $\spt\mu$.

\item For a probability  $\gamma \in \PP(\Rd \times \ldots \times \Rd)$ on the product of $n$ ambient spaces, by $\gamma_{k_1,\ldots,k_m}$ we understand the marginal $\Pi_{k_1,\ldots,k_m}^\#(\gamma)$ where, for $m\leq n$,  $\Pi_{k_1,\ldots,k_m}$ is the projection onto the coordinates ${k_1,\ldots,k_m}$.

 \item Given $\mu, \nu\in \PP(\Rd)$, we denote by $\Gamma(\mu,\nu)$ the convex subset of transport plans,
 $$ \Gamma(\mu,\nu):= \Big\{ \gamma\in \PP(\R^d \times \Rd)\ :\   \Pi_1^\# \gamma=\mu, \ \ \Pi_2^\#\gamma=\nu \Big\}.$$

 \item $[\mu]$ stands for the barycentre of a probability $\mu\in \PP_1(\Rd)$, whilst its variance  $\var(\mu):= \int \abs{x-[\mu]}^2 \mu(dx)$ is finite if and only if $\mu\in\PP_2(\Rd)$. By $\sigma_\mu := \sqrt{\var(\mu)}$ we will understand the standard deviation.

\item $\nu \succeq_c \mu$ denotes the convex order between  two probability measures $\mu,\nu\in \PP_1(\Rd)$, whilst $\Gamma_{\mathrm{M}}(\mu,\nu) \subset \PP(\Rd \times \Rd)$ is the set of martingale transport plans from $\mu$ to $\nu$.

\item Assume $\mu\in \PP(\R^n)$ and a map $x \mapsto \lambda^x \in \PP(\R^m)$ that is $\mu$-measurable in the sense that $x \mapsto \lambda^x(A)$ is $\mu$-measurable for any Borel set $A \subset \R^m$. We will use the notation,
\begin{equation}
	\label{eq:convention} \gamma= \mu \otimes \lambda^x,
\end{equation}
to define a probability $\gamma \in \PP(\R^n \times \R^m)$  that satisfies,
\begin{align*}
	 \gamma(B):= \int  \left(  \int  \chi_{B}(x,y)\,\lambda^x(dy) \right) \mu(dx),
\end{align*}
for every Borel set  $B \subset \R^n \times \R^m$, where $\chi_B$ is the characteristic function.



%
%
\end{itemize}

\section{Preliminary results}

In this section, we provide an overview of the duality result in \cite{bolbou2024}, which exposes the connection between the second-order Zolotarev distance and a new optimal transport formulation with martingale conditions.
Let us recall that the definition of the second-order Zolotarev distance between two elements  $\mu,\nu \in \mathcal{P}_2(\Rd)$ such that $[\mu] = [\nu]$,
\begin{equation}
\label{Zolotarev}
Z_2(\mu,\nu):= \sup\left\{  \int u \, d(\nu-\mu) \ :\  u\in C^{1,1}(\R^d),\ \  \Lip(\nabla u)\le 1\right\}.
\end{equation}

\begin{remark}
\label{finiteZ}
Note that the above supremum is infinite if the condition  $[\mu] = [\nu]$ is violated.  To avoid this, it is common to add constraints to $u$, namely
that $u(x_0)=0$ and $\nabla u(x_0)=0$,  where $x_0$ is usually the origin (see \cite{belili2000}). While this modified distance is well-defined on the whole set $\PP_2(\R^d)$, it is no longer translation-invariant, which complicates the search for universal inequalities.
\end{remark}

Given a pair  $(\mu,\nu)\in (\PP(\R^d))^2$ such that $[\mu] = [\nu]$, we are going to define a minimization problem in duality with the maximization problem \eqref{Zolotarev}, in which the unknown is a three marginal probability measure $\pi\in  \PP_2((\Rd)^3)$ (in short a  \emph{3-plan}), subject to linear constraints defining the set  $\Sigma(\mu,\nu)\subset \PP_2((\Rd)^3)$ as below .

\subsection{The feasible set $\Sigma(\mu,\nu)$}  It consists of all 3-plans $\pi$ whose first and second marginal is
$\mu$ and $\nu$, respectively, and which satisfy the following  equations,
\begin{equation}\label{equilibrium}
	\iiint \! \pairing{z-x, \Phi(x)} \, \pi(dxdydz)  = \iiint \! \pairing{z-y, \Psi(y)} \, \pi(dxdydz) =0,
\end{equation}
for any smooth test functions $\Phi,\Psi: \Rd \to \Rd$ with compact support.
Note that, by a density argument,  the equality above immediately extends to Borel functions $\Phi, \Psi$ of linear growth, i.e. satisfying $|\Phi|,|\Psi|\le C (1+ |x|)  .$
 The relations \eqref{equilibrium} were  introduced in \cite{bolbou2024} in  the context of structural mechanics\footnote{They encode the moment equilibrium at junctions of beam systems.}
which might seem remote from our present concern.  It turns out  in fact that they also have a natural  interpretation in probability theory through convex order and martingale transport.
 For the reader's convenience, let us recall some classical facts and  definitions. For any  coupling probability  $\gamma \in \Gamma(\mu,\nu)$, there exists a unique disintegration of the form $\gamma = \mu \otimes \gamma^x$ (see the notation \eqref{eq:convention}), where $x \mapsto \gamma^x \in \PP(\Rd)$ is a $\mu$-measurable  mapping, cf. for instance Section 2.5 in \cite{ambrosio2000}. The class of martingale transport plans from $\mu$ to $\nu$ is then defined by,
\begin{equation}\label{def:martingale}
\Gamma_{\mathrm{M}}(\mu,\nu) := \Big \{ \gamma\in \Gamma(\mu,\nu) \ :\  \gamma=\mu \otimes \gamma^x \ ,\  [\gamma^x] = x   \ \ \text{for $\mu$-a.e. $x$}  \Big\} .
 \end{equation} 
  A consequence  of the celebrated Strassen theorem \cite{Strassen-1965}
 is that   $\Gamma_{\mathrm{M}}(\mu,\nu)$ is non-empty if and only if $\nu  \succeq_c \mu$ in the sense of convex order, that is for every \emph{convex} function $\varphi:\R^d \to \R$ there holds the inequality,
\begin{equation}\label{order}
\int \varphi\, d\nu \ge \int  \varphi\, d\mu .
 \end{equation}
The stochastic interpretation of the set $\Sigma(\mu,\nu)$ appears in the following result  in \cite[Lemma 3.3]{bolbou2024} for which we refer for a proof.
\begin{lemma}\label{Sigmamunu=Amunu} Let $\pi\in \PP_2((\Rd)^3)$ be a $3$-plan with marginals $(\mu,\nu,\rho)$. Define the marginals $\pi_{1,3} : =\Pi_{1,3}^\#(\pi)$ and $\pi_{2,3} :=\Pi_{2,3}^\#(\pi)$, which are the push forwards of $\pi(dxdydz)$ through the projection maps $(x,y,z)\mapsto (x,z)$ and $(x,y,z)\mapsto (y,z)$, respectively.  
Then,
\begin{equation}\label{equival}
\pi\in \Sigma(\mu,\nu) \quad \Leftrightarrow \quad  \begin{cases} \pi_{1,3} \in \Gamma_{\mathrm{M}}(\mu,\rho), \\
  \pi_{2,3} \in \Gamma_{\mathrm{M}}(\nu,\rho). \end{cases}
\end{equation}
Accordingly, an element $\rho\in\PP_2(\Rd)$ is the third marginal of a 3-plan $\pi \in \Sigma(\mu,\nu)$  if and only if,
\begin{equation}\label{rho-admi}
\rho  \succeq_c \mu \quad \text{and}\quad \rho  \succeq_c \nu .
 \end{equation}
\end{lemma}

 \subsection{The duality result}
 In \cite{legruyer2009} it was evidenced that a $C^1$ function $u$ satisfies the constraint $\mathrm{lip}(\nabla u) \leq 1$ if and only if the following 3-point inequality holds true,
\begin{align}
	\label{3-point_ineq}
	\big({u}(y)+ \pairing{\nabla {u}(y),z-y}\big) -  \big({u}(x)  +  \pairing{\nabla {u}(x),z-x} \big) \leq \frac{1}{2} \Big( \abs{z-x}^2 + \abs{z-y}^2\Big)  \qquad 
	\forall\,(x,y,z) \in \Rd,
\end{align}
see also Lemma 3.4 in \cite{bolbou2024} for an independent proof. Accordingly, 
for every $u$ admissible in \eqref{Zolotarev} and every  3-plan $\pi \in \Sigma(\mu,\nu)$, we find that,
\begin{align}
	\nonumber
	\int u \,d(\nu-\mu) &= \iiint \bigg(	\big({u}(y)+ \pairing{\nabla {u}(y),z-y}\big) -  \big({u}(x)  +  \pairing{\nabla {u}(x),z-x} \big)  \bigg) \pi(dxdydz) \\
	\label{chain}
	& \leq \iiint \frac{1}{2} \Bigl( \abs{z-x}^2 + \abs{z-y}^2 \Bigr) \, \pi(dxdydz).
\end{align}
The equality above follows by testing \eqref{equilibrium} with $\Phi = \Psi = \nabla u$, and also by the fact that  $\mu$ and $\nu$ are the first and, respectively, the second marginals of every feasible plan $\pi$. Taking the supremum of the left hand side with respect to admissible potentials $u$, we find that each 3-plan furnishes an upper bound on the second Zolotarev distance,
\begin{align*}
	 Z_2(\mu,\nu) \leq  \iiint \frac{1}{2} \Bigl( \abs{z-x}^2 + \abs{z-y}^2 \Bigr) \, \pi(dxdydz)  \qquad\quad \forall\, \pi\in \Sigma(\mu,\nu).
\end{align*}
Our new duality scheme \cite{bolbou2024}  shows that for at least one feasible 3-plan $\pi = \ov\pi$ the gap vanishes. As in the case of the classical first-order Kantorovich-Rubinstein duality $Z_1(\mu,\nu) = W_1(\mu,\nu)$, the proof of this fact goes well beyond the standard use of the convex duality tools. It leads through the optimal convex dominance problem briefly recapitulated in Remark \ref{stochastic-opt} below. One of the main techniques relies on the fact that convexification  preserves the semi-concavity.

For the details on the second-order variant of the Kantorovich-Rubinstein duality we refer to \cite{bolbou2024}. Below, we state the result along with the optimality criteria.

\begin{theorem}
    \label{BoBo}
    For any pair $\mu,\nu \in \mathcal{P}_2(\Rd)$ such that $[\mu] = [\nu]$, the following statements hold true:
    
  \noindent $(i)$ \  one has the equality,
        \begin{equation}\label{OT3}
            Z_2(\mu,\nu) =  \inf\ 
            \biggl\{  \iiint \frac{1}{2} \Bigl( \abs{z-x}^2 + \abs{z-y}^2 \Bigr) \, \pi(dxdydz)  \ :   \ \pi\in \Sigma(\mu,\nu)\biggr\},
        \end{equation} 
        where the infimum  is achieved   at least at one 3-plan  $\ov{\pi}$;
        
    \med
    $(ii)$\  the maximization problem \eqref{Zolotarev} has a solution $\ov{u}$;
       
       \med
       $(iii)$\ the 3-plan $\ov{\pi} \in \Sigma(\mu,\nu)$ and the function $\ov{u} \in C^{1,1}(\Rd)$ with $\mathrm{lip}(\nabla \ov{u}) \leq 1$ solve the problems  \eqref{OT3} and \eqref{Zolotarev}, respectively, if and only if the two following conditions are met:
        \begin{itemize}
            \item there exists a  transport plan $\ov\gamma \in \Gamma(\mu,\nu)$ such that $\ov{\pi}$ satisfies 
             \begin{equation}
             	\label{disintegrated}
                \ov\pi(dxdydx) = \ov\gamma(dxdy) \otimes \delta_{\ov{z}(x,y)}(dz),
            \end{equation}
            where
            \begin{equation}
                 z_{\ov u}(x,y) := \frac{x+y}{2} + \frac{\nabla \ov u(y)-\nabla \ov u(x)}{2};
            \end{equation}
            \item $\ov u$ satisfies the  two-point equality  for $\ov\gamma$-a.e. pair $(x,y)$:
            \begin{align}\label{2-eq}
            \ov{u}(y)-\ov{u}(x) = \frac{1}{2} \pairing{\nabla \ov{u}(x) + \nabla \ov{u}(y), y-x} - \frac{1}{4} |\nabla \ov{u}(x) - \nabla \ov{u}(y)|^2 \, + \, & \frac{1}{4} \abs{x-y}^2 .
                   \end{align}
        \end{itemize}
\end{theorem}

\begin{proof} The first two statements are a rephrasing of the assertion $(i)$ of  \cite[Theorem 1.1]{bolbou2024}. Furthermore, the optimality conditions appearing in our  assertion $(iii)$ can be  deduced directly from  the assertion $(ii)$ of the latter theorem and of the subsequent Corollary 1.2 in the same paper \cite{bolbou2024}.
Nonetheless, with the equality \eqref{OT3} at hand, a short proof of the conditions $(iii)$ can be devised without further reference to \cite{bolbou2024}, which we will now do for the reader's convenience.

From $(i)$ we see that the optimality of an admissible pair $(\ov{u},\ov{\pi})$ is equivalent to the inequality in \eqref{chain} being an equality. In turn, this is equivalent to the equality sign in the inequality \eqref{3-point_ineq} for every triple $(x,y,z)$ in the support of $\ov\pi$. It is straightforward to check that, for fixed $x,y$, the point $z = z_{\ov{u}}(x,y)$ is the unique minimizer of the gap in the inequality \eqref{3-point_ineq}. Therefore, the disintegrated form of $\ov\pi$ in \eqref{disintegrated} is a necessary condition for the optimality. Once this form is enforced, we must make sure that the inequality \eqref{3-point_ineq}, with $z$ substituted by $z_{\ov u}(x,y)$, is an equality for all  $(x,y) \in \spt\ov\gamma$. After elementary rearrangements, this equality can be identified as \eqref{2-eq}. We have thus showed that the conditions $(iii)$ are necessary, and similarly we deduce that they are also sufficient.
\end{proof}

\begin{remark} \label{stochastic-opt} Thanks to the characterization \eqref{rho-admi} of the third marginals associated with  admissible 3-plans, one infers that the optimal transport problem \eqref{OT3} can be reduced to a stochastic optimization problem under convex dominance constraints. Indeed, by developing the three-point cost and  testing  \eqref{equilibrium} with the identity maps, one deduces  the following equality,
$$ Z_2(\mu,\nu) \, =\, \V(\mu,\nu) - \frac1{2} \big(\var (\mu) +\var (\nu) \big), \qquad \V(\mu,\nu):= \inf \Big\{ \var (\rho) \, :\, \rho  \succeq_c \mu, \ \  \rho  \succeq_c \nu \Big\}.$$
In turn, the problem  $\V(\mu,\nu)$ above admits optimal solutions $\ov\rho$ from which optimal 3-plans $\ov\pi$ can be constructed by gluing \textit{any} martingale transports in $\Gamma_{\mathrm{M}}(\mu,\ov\rho)$ and $\Gamma_{\mathrm{M}}(\nu,\ov\rho)$.
For further details, we refer to \cite{bolbou2024}. 
Note that if $\nu  \succeq_c \mu$, the unique minimizer for $\V(\mu,\nu)$ is trivially  $\ov\rho=\nu$, while a solution to \eqref{Zolotarev} is given by $\ov u= \frac1{2} |x|^2$. In particular, we have $Z_2(\mu,\nu) = \frac1{2} \big(\var (\nu) -\var (\mu) \big)$ in this case.
 
\end{remark}

\subsection{An auxiliary identity}  We will now derive an alternative expression for the Zolotarev-2 distance $Z_2(\mu,\nu)$ which was not presented in the paper \cite{bolbou2024}. This expression is crucial for establishing the optimal upper bound
 \eqref{new-upper} announced in the introduction.  
 

\begin{lemma}
    \label{magic_formula}
    Let $\ov{u}$ be a solution of the problem $Z_2(\mu,\nu)$, see \eqref{Zolotarev}. Then, the following formula holds true,
    \begin{align}\label{magic=}
        Z_2(\mu,\nu) &= \int \frac{1}{2} \pairing{x,\nabla \ov{u}(x)} \, (\nu-\mu)(dx).
    \end{align}
\end{lemma}
\begin{proof}
    Assume a solution $\ov{\pi} \in \Sigma(\mu,\nu)$. Owing to Theorem  \ref{BoBo}, it is of the form $\ov{\pi}(dxdydz) = \ov\gamma(dxdy) \otimes \delta_{\ov{z}(x,y)}(dz)$, where $\ov\gamma \in \Gamma(\mu,\nu)$ and $\ov{z}(x,y) = \frac{x+y}{2} + \frac{\nabla \ov u(y)-\nabla \ov u(x)}{2}$. We get,
    \begin{align}
        Z_2(\mu,\nu) &= \iiint \frac{1}{2} \Bigl( \abs{z-x}^2 + \abs{z-y}^2 \Bigr) \, \ov\pi(dxdydz)  \\
        & = \iint \frac{1}{2} \Bigl( \abs{\ov{z}(x,y)-x}^2 + \abs{\ov{z}(x,y)-y}^2 \Bigr) \, \ov\gamma(dxdy).
    \end{align}
    Thanks to the martingale conditions defining the set $\Sigma(\mu,\nu)$, for any Borel function $\Phi:\Rd \to \Rd$ of linear growth it holds that,
    \begin{equation} \label{mart1}
        0 = \iiint  \pairing{z-x,\Phi(x)} \, \ov\pi(dxdydz) = \iint  \pairing{\ov{z}(x,y)-x,\Phi(x)} \, \ov\gamma(dxdy),
    \end{equation}
    and similarly we get,
    \begin{equation} \label{mart2}
       \iint  \pairing{\ov{z}(x,y)-y,\Psi(y)} \, \ov\gamma(dxdy)=0.
    \end{equation}
  We then deduce the following chain of equalities,
    \begin{align*}
    	Z_2(\mu,\nu) & =  \iint \frac{1}{2} \Bigl( \pairing{\ov{z}(x,y)-x,\ov{z}(x,y)-x} + \pairing{\ov{z}(x,y)-y,\ov{z}(x,y)-y} \Bigr) \, \ov\gamma(dxdy) \\
    	& =  \iint \frac{1}{2} \Bigl( \pairing{\ov{z}(x,y)-x,\ov{z}(x,y)} + \pairing{\ov{z}(x,y)-y,\ov{z}(x,y)} \Bigr) \, \ov\gamma(dxdy) \\
    	&= \iint \Big\langle \ov{z}(x,y) - \frac{x+y}{2}, \ov{z}(x,y) \Big\rangle  \, \ov\gamma(dxdy) = \iint \Big\langle \frac{ \nabla u(y) - \nabla u(x)}{2}, \ov{z}(x,y) \Big\rangle  \, \ov\gamma(dxdy) \\
    	& = \frac{1}{2 } \iint \big\langle \nabla u(y), \ov{z}(x,y) \big\rangle  \, \ov\gamma(dxdy)  - \frac{1}{2 } \iint \big\langle \nabla u(x), \ov{z}(x,y) \big\rangle  \, \ov\gamma(dxdy) \\
    	& =  \frac{1}{2 } \iint \big\langle \nabla u(y),x \big\rangle  \, \ov\gamma(dxdy)  - \frac{1}{2 } \iint \big\langle \nabla u(x), y \big\rangle  \, \ov\gamma(dxdy) ,
    \end{align*}
    where:
    \begin{enumerate}
    \item[-] to pass from the second to the third line, we used
     \eqref{mart1},\,\eqref{mart2}  with $\Phi(x) = x$, $\Psi(y) = y$;
     \item[-] to pass to the last line, we used \eqref{mart1},\,\eqref{mart2}  with $\Phi(x) = \nabla u(x)$ and $\Psi(y) = \nabla u(y)$.
       \end{enumerate} 
    The asserted equality \eqref{magic=} now follows since $\ov\gamma \in \Gamma(\mu,\nu)$.
\end{proof}


\section{The lower bound inequality}

We  now present a sharp extension of Rio inequality to $\R^d$. This inequality  involves the ratio  $\dis \tfrac{Z_2(\mu,\nu)}{W_2^2(\mu,\nu)}$,
which is  dilation-invariant since the Wassertein distance enters with a square.
\begin{theorem}    \label{main}
    Let  $\mu,\nu$ be two probability measures in  $\mathcal{P}_2(\Rd)$. Then,
	\begin{equation}
		\label{lower_bound}
		\frac{1}{4} W_2^2(\mu,\nu)  \leq  Z_2(\mu,\nu),
	\end{equation}
	and the constant $\frac{1}{4}$  cannot be improved. Moreover, the inequality is an equality if and only if $\mu = \nu$. 
\end{theorem}

Let us remark that the inequality \eqref{lower_bound} is trivially true if $[\mu]\not=[\nu]$. Indeed, in this case the supremum in  
\eqref{Zolotarev} is infinite and, by convention, $Z_2(\mu,\nu)=+\infty.$

\begin{proof}[Proof of the lower bound] 
    It is enough to observe that for any triple $(x,y,z) \in (\Rd)^3$ we have,
    \begin{equation}
        \frac{1}{2} \Bigl( \abs{z-x}^2 + \abs{z-y}^2 \Bigr) \geq \frac{1}{4} \abs{x-y}^2,
    \end{equation}
    while the equality holds if and only if $z=\frac{x+y}{2}$. Let  $\ov\pi$ be a solution of the three-point optimal transport formulation \eqref{OT3} and denote $\ov\gamma := \Pi_{1,2}^\# \ov\pi$. Then, by the assertion (i) of Theorem \ref{BoBo},
    \begin{align}
    	\label{chain_lower}
        Z_2(\mu,\nu) = \iiint \frac{1}{2} \Bigl( \abs{z-x}^2 + \abs{z-y}^2 \Bigr) \, \ov\pi(dxdydz) &\geq \frac{1}{4} \iiint \abs{x-y}^2 \, \ov\pi(dxdydz) \\
        \nonumber
        & = \frac{1}{4}\iint \abs{x-y}^2 \, \ov\gamma(dxdy) \\
        \nonumber
        & \geq \frac{1}{4} \,W_2^2(\mu,\nu).
    \end{align}
    To obtain the last inequality, we used the admissibility condition $\ov\pi \in \Sigma(\mu,\nu)$ which entails that $\ov\gamma \in \Gamma(\mu,\nu)$. 
    The optimality of the constant $\frac{1}{4}$ follows from  the Example \ref{opt1/4} below.
    
    Let us now pass to examining the criteria for \eqref{lower_bound} being an equality. Clearly the condition $\mu = \nu$ is sufficient. To show the necessity, we observe that the equality implies that the first inequality in \eqref{chain_lower} is an equality. In turn, this means that,
    \begin{equation}
    	\ov{\pi}(dxdydz) = \ov{\gamma}(dxdy) \otimes \delta_{\frac{x+y}{2}} (dz).
    \end{equation}
    The two martingale conditions \eqref{equilibrium} defining the set $\Sigma(\mu,\nu)$ furnish the relations,
    \begin{align*}
    	&0 = \iiint \pairing{\Phi(x),z-x} \,\ov\pi(dxdydz) = \frac{1}{2} \iint  \pairing{\Phi(x),y-x} \, \ov{\gamma}(dxdy), \\
    	&0 = \iiint \pairing{\Psi(y),z-y} \,\ov\pi(dxdydz) = \frac{1}{2} \iint  \pairing{\Psi(y),x-y} \, \ov{\gamma}(dxdy),
    \end{align*}
    where   $\Phi,\Psi: \Rd \to \Rd$ are any Borel maps of linear growth.
    By  adding the two lines above, after taking $\Phi(x) = x$ and $\Psi(y)=y$, we infer that $ \iint |y-x|^2 \, \ov{\gamma}(dxdy) =0$. 
    Therefore $\ov \gamma$ must be supported on the diagonal $\Delta=\{(x,x) : x\in\Rd\}$. It is thus straightforward that its marginals $\mu$ and $\nu$ must coincide. 
    The proof is complete.
\end{proof}

\begin{example}
    \label{opt1/4}
    For $b>a>0$ consider the two centred probabilities on the real line:
    \begin{equation*}
        \mu_{a,b} = \frac{b}{a+b}\, \delta_{-a} + \frac{a}{a+b} \,\delta_b, \qquad \nu_{a,b} = \frac{a}{a+b} \,\delta_{-b} + \frac{b}{a+b} \, \delta_a.
    \end{equation*}
    The unique transport plan with a monotone support reads,
    \begin{equation*}
        \gamma_{\mathrm{mon}} = \frac{a}{a+b} \, \delta_{(-a,-b)} + \frac{b-a}{a+b}\, \delta_{(-a,a)} + \frac{a}{a+b} \, \delta_{(b,a)}.
    \end{equation*}
    It is straightforward to check that $\Pi_1^\# \gamma_{\mathrm{mon}} = \mu_{a,b}$ and $\Pi_2^\# \gamma_{\mathrm{mon}} = \nu_{a,b}$.
    Thus, the square of the Wasserstein distance equals,
    \begin{align*}
        W_2^2(\mu_{a,b},\nu_{a,b}) &= \iint \abs{x-y}^2 \gamma_{\mathrm{mon}}(dxdy) \\
        &=(b-a)^2 \frac{a}{a+b} + (2a)^2 \frac{b-a}{a+b} + (b-a)^2 \frac{a}{a+b} = 2a (b-a).
    \end{align*}
    In turn, let us propose an admissible 3-plan $\ov\pi \in \Sigma(\mu_{a,b},\nu_{a,b})$,
    \begin{equation*}
        \ov\pi = \frac{a}{a+b} \, \delta_{(-a,-b,-b)} + \frac{b-a}{a+b} \, \delta_{(-a,a,0)} + \frac{a}{a+b} \, \delta_{(b,a,b)}.
    \end{equation*}
    Since $\Pi_{1,2}^\# \ov\pi = \gamma_{\mathrm{mon}}$, to show the admissibility of $\ov\pi$ it remains to check the martingale conditions. We compute,
    \begin{align*}
        \Pi_{1,3}^\# \ov\pi &= \frac{a}{a+b} \, \delta_{(-a,-b)} + \frac{b-a}{a+b} \, \delta_{(-a,0)} + \frac{a}{a+b} \, \delta_{(b,b)} \\
        & = \frac{b}{a+b} \, \delta_{-a} \otimes  \bigg( \frac{a}{b} \,\delta_{-b} + \frac{b-a}{b} \, \delta_{0} \bigg) +  \frac{a}{a+b} \, \delta_{b} \otimes \delta_b.
    \end{align*}
    Since the probability $\frac{a}{b} \,\delta_{-b} + \frac{b-a}{b} \, \delta_{0}$ has $-a$ as its barycentre, it is easy to see that the above plan is a martingale. Similarly, one shows that $\Pi_{2,3}^\# \ov\pi$ is a martingale as well. We can readily compute an upper bound for the Zolotarev distance,
    \begin{align*}
        Z_2(\mu_{a,b},\nu_{a,b}) &\leq \iiint \frac{1}{2} \Bigl( \abs{z-x}^2 + \abs{z-y}^2 \Bigr) \, \ov\pi(dxdydz) \\
        &=  \frac{1}{2} (-b+a)^2 \frac{a}{a+b} + \frac{1}{2} \big(a^2 + a^2\big) \,\frac{b-a}{a+b} +  \frac{1}{2} (b-a)^2 \frac{a}{a+b} = \frac{ab}{a+b}\,(b-a).
    \end{align*}
    With $a>0$ fixed, we look at the asymptotic quotient when $b$ converges to $a$ from above, 
    \begin{equation*}
        \limsup_{b \,\searrow \, a} \frac{Z_2(\mu_{a,b},\nu_{a,b})}{W_2^2(\mu_{a,b},\nu_{a,b})} \leq \lim_{b \,\searrow \, a}\frac{b}{2(a+b)} = \frac{1}{4}.
    \end{equation*}
    This establishes the optimality of the constant $\frac{1}{4}$ in the lower bound in Theorem \ref{main} in dimension $d=1$. Clearly this example can be embedded in $\R^d$
    by taking pairs of symmetric  points on a straight line.  Hence it holds for every $d\ge 1$ that,
     $$\inf \left\{\frac{Z_2(\mu,\nu) }{W_2^2(\mu,\nu)}\, :\, \mu,\nu \in \PP_2(\R^d) \ ,\ \mu\not=\nu\right\}= \frac1{4}.$$ 
\end{example}

\begin{remark}\label{noreverse}  There does not exist a reverse  inequality of the kind $Z_2(\mu,\nu) \le C\, W_2^2(\mu,\nu)$ holding for pairs $(\mu,\nu)$ such that
$[\mu]=[\nu]$. This can be checked by showing that the ratio $\frac{Z_2(\mu,\nu_n) }{W^2_2(\mu,\nu_n)}$ blows up when $\mu:= \frac1{2} (\delta_1+ \delta_{-1})$ and $\nu_n:= \frac1{2} (\delta_{1+1/n}+ \delta_{-1-1/n})$.
Indeed, in that case, we have $W^2_2(\mu,\nu_n)=1/n^2$, while, in view of  Remark \ref{stochastic-opt}, it holds that $Z_2(\mu,\nu_n)=\frac1{2} \big(\var (\nu_n) - \var(\mu)\big)= \frac1{2} \big((1+1/n)^2-1\big)$, which goes to zero at the rate $1/n$. This shows that a reverse inequality 
	cannot be true even if we confine ourselves to probability measures with supports contained in a fixed compact set $K \subset \Rd$. Another counter-example, but with non-compact supports, is given below for the Gaussian laws.

\end{remark}

    \begin{example}
    \label{opt1/2}
    For two positive numbers $\sigma_1 \leq \sigma_2$ consider two centred Gaussians on the real line with standard deviations equal to $\sigma_1$ and $\sigma_2$, respectively, that is $\mathcal{N}(0,\sigma^2_1)$, $\mathcal{N}(0,\sigma^2_2)$.
    The Wasserstein-2 distance between two Gaussian distributions admits a closed form in any dimension $d$, see e.g. \cite{dowson1982}. In the simple case of two 1D centered Gaussians, the formula reduces to,
    \begin{equation}
        W_2\Big(\mathcal{N}(0,\sigma^2_1) \, ,\, \mathcal{N}(0,\sigma^2_2)\Big) = \abs{\sigma_1-\sigma_2} = \sigma_2 -\sigma_1.
    \end{equation}
    It is also well known that every pair of centred Gaussians  in 1D is in convex order. More precisely, in case when $\sigma_1 \leq \sigma_2$, we have,
    \begin{equation*}
        \mathcal{N}(0,\sigma^2_1) \, \preceq_c \, \mathcal{N}(0,\sigma^2_2).
    \end{equation*}
    Meanwhile, for any  measures in convex order $\mu \preceq_c \nu$, the Zolotarev distance equals $Z_2(\mu,\nu) = \frac{1}{2} \big( \var(\nu) - \var(\mu) \big)$, cf. \cite[Section 4.1]{bolbou2024}. Accordingly, for the Gaussian measures we have,
    \begin{equation}
        Z_2\Big(\mathcal{N}(0,\sigma^2_1) \, ,\, \mathcal{N}(0,\sigma^2_2)\Big) = \frac{\sigma_2^2 - \sigma_1^2}{2}. 
    \end{equation}
  By considering  $\mu:=\mathcal{N}(0,1)$ and $\nu_n:=\mathcal{N}(0,1+ \frac1{n})$,  we arrive at: 
        \begin{equation*}
        \frac{ Z_2(\mu,\nu_n)}{ W_2^2(\mu,\nu_n)}\,  =\, \frac{2n +1}{2}\, \to \, +\infty .
    \end{equation*}
\end{example}

\vskip1cm

\section{Upper bound inequalities}

From now on, we consider pairs $\mu,\nu \in \mathcal{P}_2(\Rd)$ such that $[\mu] = [\nu]$. Our goal is to derive an upper bound for the ratio $\frac{ Z_2(\mu,\nu)}{ W_2(\mu,\nu)}$
as a function of the variances of $\mu$ and $\nu$. More precisely, we want to establish an explicit formula for the function $h(a,b): \R_+^2 \to \R_+$ given by,
\begin{equation}\label{h-function}
h(a,b) : = \sup  \left\{ \frac{ Z_2(\mu,\nu)}{ W_2(\mu,\nu)} \, :\,   \sigma_\mu \le a, \ \sigma_\nu\le b, \ \mu \neq \nu , \  [\mu]=[\nu]=0 \right\},
\end{equation}
where $\sigma_\mu = \sqrt{\var(\mu)}$, $\sigma_\nu = \sqrt{\var(\nu)}$ are the standard deviations.
 By acknowledging how the two distances scale under dilations $T_\lambda(x) =\lambda x$, it is straightforward to check the homogeneity property  $h(\lambda a, \lambda b) = \lambda \, h(a,b)$ holding for every $\lambda>0$.
On the other hand, if $a=0$ (thus $\mu=\delta_0$), we get $Z_2(\mu,\nu) = \frac1{2} \sigma_\nu^2 = \frac1{2} b^2$, while $W_2(\mu,\nu)= \sigma_\nu =  b$. Hence $h(0,b)=h(b,0)= \frac{b}{2}$.
We are going to show that,  in fact, $h$ is linear, i.e. $h(a,b)= \frac1{2} (a+b)$. This will give the optimal upper bound stated below.

\begin{theorem}
	\label{thm_upper}
	Assume that the two probability distributions $\mu,\nu \in \mathcal{P}_2(\Rd)$ share the barycentre, i.e. $[\mu] = [\nu]$. Then, we have,
	\begin{equation}
		\label{upper_bound}
		 Z_2(\mu,\nu)  \leq \frac{1}{2}  (\sigma_\mu + \sigma_\nu)\, W_2(\mu,\nu) ,
	\end{equation}
	and this upper bound is optimal in the sense that  $h(a,b)= \frac1{2}(a+b)$ in \eqref{h-function}.
	
	Furthermore, the inequality becomes an equality if and only if  the probability measures coincide up to a dilation centred at their common barycentre.
\end{theorem}

Before giving the proof which is postponed to the end of this section, we present a  straightforward variant of the inequality \eqref{upper_bound}
which, in constrast, is strict whenever $\mu\not=\nu$.
\begin{corollary}\label{variant_upper}  Under the assumptions of Theorem \ref{thm_upper}, we have the following upper bound,
\begin{equation}
		\label{varupper}
		 Z_2(\mu,\nu) \, \leq \ \sqrt{\frac{\var\, \mu + \var\, \nu}{2}}\ W_2(\mu,\nu) ,
	\end{equation}
where the inequality is an equality if and only if $\mu=\nu$. 
\end{corollary}

\begin{proof} The inequality \eqref{varupper} follows from \eqref{upper_bound} and from the inequality $\frac {a+b}{2} \le \sqrt{\frac{a^2 + b^2}{2}}$.
 Noticing that the latter inequality is strict if $a\not=b$, we infer that an equality in \eqref{varupper} implies the equality in 
 \eqref{upper_bound} along with  $\sigma_\mu=\sigma_\nu$. Owing to the last statement of Theorem \ref{thm_upper}, we conclude that $\mu=\nu$  since the only possible dilation sending $\mu$ to $\nu$ is the identity.
\end{proof}

\begin{remark}\label{equivalence} From Theorem \ref{main} and Theorem \ref{thm_upper}, we infer that on the subset  of centered probabilities  with finite second order moment,
    \begin{equation*}
        \mathcal{P}_2^{(0)}(\Rd) := \Big\{ \mu \in \mathcal{P}_2(\Rd) \, : \, [\mu] =0 \Big\},
    \end{equation*}
    the topologies induced by the $W_2$ and the $Z_2$ distances are equivalent. We thus recover the result of \cite{belili2000}. Clearly, by the invariance with respect to translations, the same statement holds true if we enforce any fixed barycentre $a\in \Rd$, not necessarily the origin.
\end{remark}

\begin{proof}[Proof of Theorem \ref{thm_upper}] Without any loss of generality, we can assume that the two probabilities are centered, i.e. $[\mu] = [\nu] = 0$. Let $\ov{u}$ be a solution of the problem \eqref{Zolotarev}. As  $\mu-\nu$ vanishes on affine functions, it is not restrictive to assume that $\ov u(0)=0$ and $\nabla \ov{u}(0) = 0$. Since $\mathrm{lip}(\nabla \ov{u}) \leq 1$, this guarantees that $\abs{\nabla \ov{u}(x)} \leq \abs{x}$ for any $x \in \Rd$.

\subsection*{Step 1}[\emph{proving the upper bound and the sufficient condition for equality}]
Let us  take an optimal plan $\hat{\gamma}\in \Gamma(\mu,\nu)$ realizing the distance $W_2(\mu,\nu)$. Thanks to Lemma \ref{magic_formula},  we obtain,
\begin{align}
	\label{integral_chain}
	2 Z_2(\mu,\nu) & =\iint \Big( \pairing{y,\nabla \ov{u}(y)} - \pairing{x,\nabla \ov{u}(x)} \Big) \hat{\gamma}(dxdy) \\
	\nonumber
	& =\iint \Big(\frac{1}{2} \pairing{y-x,\nabla \ov{u}(x) + \nabla \ov{u}(y)} + \frac{1}{2} \pairing{x+y,\nabla \ov{u}(y) - \nabla \ov{u}(x)}\Big) \hat{\gamma}(dxdy) \\
	\nonumber
	& =\frac{1}{2}\iint  \pairing{\nabla \ov{u}(x),y-x}  \, \hat{\gamma}(dxdy) + \frac{1}{2}\iint  \pairing{\nabla \ov{u}(x),y-x} \, \hat{\gamma}(dxdy) \\
	\nonumber
	& \qquad + \frac{1}{2}\iint  \pairing{x,\nabla \ov{u}(y) - \nabla \ov{u}(x)} \, \hat{\gamma}(dxdy) + \frac{1}{2}\iint  \pairing{y,\nabla \ov{u}(y) - \nabla \ov{u}(x)} \, \hat{\gamma}(dxdy)   \\
	\label{ineq1}
	& \leq \frac{1}{2} \left(\iint \abs{\nabla \ov{u}(x)}^2 \hat\gamma(dxdy) \right)^{\frac{1}{2}} \left(\iint \abs{y-x}^2 \hat\gamma(dxdy) \right)^{\frac{1}{2}} \\
	\nonumber
	& \qquad + \frac{1}{2} \left(\iint \abs{\nabla \ov{u}(y)}^2 \hat\gamma(dxdy) \right)^{\frac{1}{2}} \left(\iint \abs{y-x}^2 \hat\gamma(dxdy) \right)^{\frac{1}{2}} \\
	\nonumber
	& \qquad \qquad  + \frac{1}{2} \left(\iint \abs{x}^2 \hat\gamma(dxdy) \right)^{\frac{1}{2}} \left(\iint \abs{\nabla \ov{u}(y) - \nabla \ov{u}(x)}^2 \hat\gamma(dxdy) \right)^{\frac{1}{2}} \\
	\nonumber
	& \qquad \qquad \qquad  + \frac{1}{2} \left(\iint \abs{y}^2 \hat\gamma(dxdy) \right)^{\frac{1}{2}} \left(\iint \abs{\nabla \ov{u}(y) - \nabla \ov{u}(x)}^2 \hat\gamma(dxdy) \right)^{\frac{1}{2}} \\
		\label{ineq2}
	& \leq \left(\int \abs{x}^2 \mu(dx) \right)^{\frac{1}{2}} \left(\iint \abs{y-x}^2 \hat\gamma(dxdy) \right)^{\frac{1}{2}} \\
	\nonumber
	& \qquad +\left(\int \abs{y}^2 \nu(dy) \right)^{\frac{1}{2}} \left(\iint \abs{y-x}^2 \hat\gamma(dxdy) \right)^{\frac{1}{2}} \\
	\nonumber
	& = \sigma_\mu W_2(\mu,\nu) + \sigma_\nu W_2(\mu,\nu).
\end{align}
Above, the first inequality \eqref{ineq1} uses  Cauchy-Schwarz inequality in $L^2_{\hat{\gamma}}(\Rd \times \Rd;\Rd)$ four times. To obtain the second inequality \eqref{ineq2} we used the fact that  $\abs{\nabla \ov{u}(x)} \leq \abs{x}$ and $\abs{\nabla \ov{u}(y) - \nabla \ov{u}(x)} \leq \abs{y-x}$. Eventually we  pass to the last line by using the optimality of $\hat\gamma$ and  the fact that $\mu, \nu$ are assumed to be centred. This proves the upper bound inequality \eqref{upper_bound} and the fact that $h$ given by \eqref{h-function} satisfies  $h(a,b)\le \frac1{2} (a+b)$.

Next, in order to show that  \eqref{upper_bound} is optimal in the sense that $h(a,b)= \frac1{2} (a+b)$, we check that the equality is saturated by all centred pairs $(\mu,\nu)$  
such that $\mu$ and $\nu$ coincide up to a dilation. Without any loss of generality, we can assume that $\nu = T_\lambda^\# \mu = (\lambda x) ^\# \mu $  with $\lambda \ge1$. Then, it is easy to see that $\mu \preceq_c \nu$ \footnote{by Strassen theorem, the probability kernel $\gamma^x  = \frac{\lambda-1}{\lambda} \nu + \frac{1}{\lambda} \delta_{\lambda x}$ renders $\gamma = \mu \otimes \gamma^x$ an element of $\Gamma_{\mathrm{M}}(\mu,\nu)$}. Accordingly, as pointed out in Remark \ref{stochastic-opt}, we have  $Z_2(\mu,\nu) = \frac{1}{2}\big(\sigma_\nu^2 - \sigma^2_\mu\big) = \frac{1}{2}(\lambda^2-1) \,\sigma^2_\mu$. On the other hand, the transport map $T_\lambda =\nabla \big( \frac{\lambda}{2} \abs{\argu}^2\big)$ is a gradient of a convex function. Therefore, by Brenier theorem \cite{brenier1991},  $W_2(\mu,\nu) = (\int \abs{T_\lambda(x)-x}^2 \mu(dx))^{1/2} = (\lambda- 1)\, \sigma_\mu$. 
It follows that $Z_2(\mu,\nu) = \frac{1}{2} (\lambda+1) \,\sigma_\mu \, W_2(\mu,\nu) = \frac{1}{2} (\sigma_\mu+\sigma_\nu)\, W_2(\mu,\nu) .$

\subsection*{Step 2}[\emph{necessary condition for equality}]
Now we  start with the assumption that \eqref{upper_bound} is an equality for some pair $(\mu,\nu)$. We may assume without any loss of generality that $\sigma_\nu \ge \sigma_\mu$. Furthermore, we can exclude the cases when either $\mu = \delta_{x_0}$ or $\mu = \nu$ since both trivially satisfy the dilation condition.   Accordingly,  we  can assume that both $\sigma_\mu$ and $\sigma_\nu$ are strictly positive, and we  want to show  that  $\nu = T_\lambda^\# \mu$ for a suitable $\lambda\in (1, +\infty)$.

Going back to the chain of relations  starting from \eqref{integral_chain}, we see that the equality in \eqref{upper_bound} implies that the inequalities in \eqref{ineq1} and in \eqref{ineq2} are actually equalities. From the equality in  \eqref{ineq2}, we infer that, for $\hat{\gamma}$-a.e. $(x,y)$, we have,
\begin{equation}
	\label{norms_eq}
	\abs{\nabla \ov{u}(x)} = \abs{x}, \qquad \abs{\nabla \ov{u}(y)} = \abs{y}, \qquad  \abs{\nabla \ov{u}(y)-\nabla \ov{u}(x)} = \abs{y-x},
\end{equation}
 By the continuity of $\nabla \ov{u}$, the equalities above  hold true for every $(x,y) \in \spt \hat{\gamma}$. Since $\mu\not=\nu$, $x-y$ is not  zero as an element of
 $L^2_{\hat{\gamma}}(\Rd \times \Rd;\Rd)$,  and the same can be deduced for the  function  
 $\nabla \ov{u}(y) - \nabla \ov{u}(x)$ which shares the same norm by vitue of the third equality in \eqref{norms_eq}. 
%
  With that information at our disposal, we may identify the conditions under which the inequality  \eqref{ineq1} is an equality. The four Cauchy-Schwarz inequalities in the vector valued space $L^2_{\hat{\gamma}}(\Rd \times \Rd;\Rd)$ must be equalities, which means that for any $(x,y) \in \spt \hat{\gamma}$,
\begin{alignat}{2}
	\label{par1}
	&\nabla \ov{u} (x) = a\, (y-x), \qquad  	&&\nabla \ov{u} (y) = b \, (y-x), \\
	\label{par2}
	& x = c \, \big(\nabla \ov{u}(y) - \nabla \ov{u}(x) \big), \qquad && y = e \, \big(\nabla \ov{u}(y) - \nabla \ov{u}(x) \big),
\end{alignat}
for suitable constants $a,b,c,e \geq 0$  which \emph{do not depend} on $x,y$. The condition $\sigma_\mu>0$ implies  that the projection of $\spt \hat{\gamma}$ on the first component does no reduce to $\{0\}$. Hence, from the first equality in \eqref{par2}, we infer that  $c$ must be strictly positive.
The equalities in the sequel of the proof will hold true for each pair $(x,y)\in \spt \hat{\gamma}$. 
From \eqref{par1}, we get  $\nabla \ov{u}(y)-\nabla \ov{u}(x) = (b-a) (y-x)$. Going back to \eqref{norms_eq}, this means that $\eps := b-a  \in \{-1,1\}$. In turn, from \eqref{par2}, we get $x= c \eps  \, (y-x)$,  which we can rewrite as,
\begin{equation}
	y = \lambda x, \quad \text{where} \quad  \lambda := 1+ \frac{1}{c\eps}.
\end{equation}
Therefore, for  $\lambda$ as above, it holds  that  $\hat{\gamma} = (\mathrm{id},T_\lambda) ^\# \mu$, hence $\nu = T_\lambda^\# \mu$.
Moreover, our assumptions  $\sigma_\nu \ge\sigma_\mu$, $\mu \neq \nu$ imply that $\lambda > 1$, which gives $\epsilon=1$ and  $\mu \preceq_c \nu$ (so that an optimal potential $\ov u$ is given by $\ov u= \frac1{2} |x|^2$).
 This ends the proof. 
\end{proof}

\bibliographystyle{plain}

\bigskip

{\footnotesize

\noindent
Karol Bo{\l}botowski:\\
Faculty of Mathematics, Informatics and Mechanics, University of Warsaw\\
2 Banacha Street, 02-097 Warsaw - POLAND\\
{\tt k.bolbotowski@mimuw.edu.pl}

\medskip

\noindent
	Guy Bouchitt\'e:\\
	Laboratoire IMATH, Universit\'e de Toulon\\
	BP 20132, 83957 La Garde Cedex - FRANCE\\
	{\tt bouchitte@univ-tln.fr}
	
}

\end{document}